\documentclass[12pt]{article}
\usepackage{amssymb,amsmath,latexsym}
\usepackage{amsthm}
\numberwithin{equation}{section}
\newtheorem{theorem}{Theorem}

\begin{document}
\author{Alexander E Patkowski}
\title{On the $q$-Pell sequences and sums of tails}
\date{\vspace{-5ex}}
\maketitle
\abstract{We examine the $q$-Pell sequences and their applications to weighted partition theorems and values of $L$-functions. We also put them 
into perspective with sums of tails.}

\section{Introduction}
In [11] we find that W. Y. Chen and K. Q. Ji offered weighted partitions theorems relating partitions into distinct parts to partitions into odd parts by paraphrasing 
some special identities due to Ramanujan (see the section on ``Special Identities" [6, pg. 149]). Namely, we have [11, eq.(4.13)]
\begin{equation}\sum_{\mu\in D}\left(\mu_1+n(\mu)+\frac{1-(-1)^{r(\mu)}}{2}\right)q^{|\mu|}=2\sum_{\lambda\in O}n(\lambda)q^{|\lambda|}.\end{equation}
Here $D$ is the set of partitions into distinct parts, and $O$ is the set of partitions into odd parts. Recall that a partition $\lambda$ of $n$ is a nonincreasing sequence $\lambda=(\lambda_1, \lambda_2, \cdots, \lambda_k),$ where the $\lambda_i,$ for $1\le i\le k,$ are parts that sum to $n.$ We denote $n(\lambda)$ as the number of parts of a partition $\lambda=(\lambda_1, \lambda_2, \cdots, \lambda_k),$ where $\lambda_1$ denotes the largest part. We denote
$r(\lambda)$ to be the largest part minus the number of parts, or the ``rank" of the partition. As it stands, equation (1.1)
is equivalent to the identity
\begin{equation}\sum_{n\ge0}\left((-q)_{\infty}-(-q)_n\right)=(-q)_{\infty}\left(-\frac{1}{2}+\sum_{n\ge1}\frac{q^n}{1-q^n}\right)+\frac{1}{2}\sum_{n\ge0}\frac{q^{n(n+1)/2}}{(-q)_n}.\end{equation}
Here the $q$-series on the far right side of (1.2) is the distinct rank parity function $\sigma(q)$ and has been studied extensively [2, 6, 9, 11, 17, 19] (see [7] for similar functions). We put $(V;q)_n:=(1-V)(1-Vq)\cdots(1-Vq^{n-1}).$ 
\par Now if we subtract $2\sum_{\mu\in D}n(\mu)q^{|\mu|}$ from both sides of (1.1) we may write
\begin{equation}\sum_{\mu\in D}\left(r(\mu)+\frac{1-(-1)^{r(\mu)}}{2}\right)q^{|\mu|}=2\sum_{\lambda\in O}n(\lambda)q^{|\lambda|}-2\sum_{\mu\in D}n(\mu)q^{|\mu|}.\end{equation}
It is not difficult to see that

\begin{equation}2(-q)_{\infty}\sum_{n\ge1}\frac{q^{2n}}{1-q^{2n}}=2\sum_{\lambda\in O}n(\lambda)q^{|\lambda|}-2\sum_{\mu\in D}n(\mu)q^{|\mu|}.\end{equation}
Hence, we have shown combinatorially the $q$-series identity [19]

\begin{equation}\sum_{n\ge0}\left((-q)_{\infty}-V_n(q)\right)=(-q)_{\infty}\left(-\frac{1}{2}+2\sum_{n\ge1}\frac{q^{2n}}{1-q^{2n}}\right)+\frac{1}{2}\sum_{n\ge0}\frac{q^{n(n+1)/2}}{(-q)_n},\end{equation}
where $V_m(q)$ is the generating function for the number of partitions of $n$ with rank $m.$ In fact, Fine [19, pg.8] has studied this function which is given by 
\begin{equation}V_m(q)=\sum_{n\ge0}{n+m \brack n}_qq^{n(n+1)/2},\end{equation}
$$\sum_{n\ge0}V_n(q)t^n=\sum_{n\ge0}\frac{q^{n(n+1)/2}}{(t)_{n+1}}.$$
Here 
\begin{equation}{n \brack m}_q:=\frac{(q)_n}{(q)_m(q)_{n-m}}.\end{equation}
The relation between (1.5) and Ramanujan's other identity related to $\sigma(q)$ [6] 
\begin{equation}\sum_{n\ge0}\left(\frac{1}{(q;q^2)_{\infty}}-\frac{1}{(q;q^2)_{n+1}}\right)=\frac{1}{(q;q^2)_{\infty}}\left(-\frac{1}{2}+\sum_{n\ge1}\frac{q^{2n}}{1-q^{2n}}\right)+\frac{1}{2}\sum_{n\ge0}\frac{q^{n(n+1)/2}}{(-q)_n}.\end{equation}
has been shown analytically in [19].
\par This study offers two more examples of two-variable Rogers-Ramanujan-type $q$-series $f(t,q)$ that satisfy the following properties:
\\*
\\*
\it
\hspace{2mm}(i) $\lim_{t\rightarrow1^{-}}(1-t)f(t,q)=f(q)$ is a weight $0$ modular form.
\\*
\hspace{2mm}(ii) $f(-1,q)$ is a false mock modular form.
\\*
\hspace{2mm}(iii) $\lim_{t\rightarrow1^{-}}\frac{\partial }{\partial t}(1-t)f(t,q)=f(q)D(q)+f(-1,q),$ where $D(q)$ is some linear combination of
generating functions for divisor functions.
\rm
\par This allows us to put sums of tails into perspective with two variable generalizations of Rogers-Ramanujan
type series [1]. The first example of a $q$-series in the literature that satisfies (i)--(iii) appears to be $\sum_{n\ge0}V_n(q)t^n,$ where for (i) we have the modular form $\eta(2z)/\eta(z),$ where as usual
we define the Dedekind eta function
$$\eta(z):=e^{2\pi iz/24}\prod_{n\ge1}^{\infty}(1-e^{2\pi inz}),$$
where $\Im(z)>0.$ \par The main theorems of this paper offer two more examples satisfying (i)--(iii), for the weight $0$ modular forms
$$\frac{\eta(4z)}{\eta(z)}, ~~~~~~~~\mbox{and}~~~~~~~~ \frac{\eta^2(2z)}{\eta^2(z)}\frac{\eta(2z)}{\eta(4z)}.$$
See, especially, [4, Theorem 3] and the commentary associated with that theorem. For more material on sums of tails see [6, 8, 16, 18, 20]. 
\section{The $q$-Pell sequences} 
In [21] Sills and Santos study the sequences
\begin{equation}\omega_n(q)=\sum_{n\ge j\ge0}\sum_{j\ge k\ge0}{j \brack k}_q{n-k \brack j}_qq^{j(j+1)/2+k(k+1)/2},\end{equation}
and
\begin{equation}\Theta_n(q)=\sum_{n\ge j\ge0}\sum_{j\ge k\ge0}{j \brack k}_q{n-k \brack j}_qq^{j(j+1)/2+k(k-1)/2},\end{equation}
where
\begin{equation}\sum_{n\ge0}\omega_n(q)t^n=\sum_{n\ge0}\frac{(-tq)_nt^nq^{n(n+1)/2}}{(t)_{n+1}},\end{equation}
and
\begin{equation}\sum_{n\ge0}\Theta_n(q)t^n=\sum_{n\ge0}\frac{(-t)_nt^nq^{n(n+1)/2}}{(t)_{n+1}}.\end{equation}
Here we shall define $L_1(t,q)$ to be (2.3), and $L_2(t,q)$ to be (2.4). That is, two-variable generalizations of the Lebesgue identities [13]. Then 
$\lim_{t\rightarrow1^{-}}(1-t)L_1(t,q)=(-q^2;q^2)_{\infty}/(q;q^2)_{\infty},$ and $\lim_{t\rightarrow1^{-}}(1-t)L_2(t,q)=(-q;q^2)_{\infty}/(q;q^2)_{\infty}$ by [13].
\par The function $L_1(-1,q)$ has appeared in [10], and $L_2(-1,q)$ has appeared in Lovejoy [14], where they have related these functions to the 
arithmetic of $\mathbb{Q}(\sqrt{2}).$ As usual, let $N(a)$ be the norm of an ideal $a\subset\mathbb{Z}[\sqrt{2}].$ [10] showed that
\begin{equation}\sum_{n\ge0}\frac{(q)_n(-1)^nq^{n(n+1)/2}}{(-q)_{n}}=\sum_{\substack{a\subset\mathbb{Z}[\sqrt{2}]\\ N(a)\equiv1\pmod{8}}}(-q)^{(N(a)-1)/8},\end{equation}
and Lovejoy [14] showed that 
\begin{equation}\sum_{n\ge1}\frac{(q)_{n-1}(-1)^nq^{n(n+1)/2}}{(-q)_{n}}=\sum_{\substack{a\subset\mathbb{Z}[\sqrt{2}]\\ N(a)<0}}i^{N(a)^2+N(a)}q^{|N(a)|}.\end{equation}
The functions (2.5) and (2.6) appear as ``error" series in the following \it sums of tails \rm identities. These are the cases [4, Theorem 1, $a=q^{3/2},$ $t=-q$] and
[4, Theorem 1, $a=q^{3/2},$ $t=-q^{1/2}$]. (Equation (2.8) follows from (2.10).)
\\*
\\*
{\bf Lemma 1.} \it We have the following sums of tails,
\begin{equation}\sum_{n\ge0}\left(\frac{(-q^2;q^2)_{\infty}}{(q;q^2)_{\infty}}-\frac{(-q^2;q^2)_{n}}{(q;q^2)_{n+1}}\right)=\frac{(-q^2;q^2)_{\infty}}{(q;q^2)_{\infty}}\left(-\frac{1}{2}+\sum_{n\ge1}\frac{q^{2n}}{1-q^{2n}}+\sum_{n\ge1}\frac{q^{2n+1}}{1+q^{2n+1}}\right)+\frac{1}{2}\sum_{n\ge0}\frac{(q)_n(-1)^nq^{n(n+1)/2}}{(-q)_{n}},\end{equation}

\begin{equation}\sum_{n\ge0}\left(\frac{(-q;q^2)_{\infty}}{(q;q^2)_{\infty}}-\frac{(-q;q^2)_{n}}{(q;q^2)_{n+1}}\right)=\frac{(-q;q^2)_{\infty}}{(q;q^2)_{\infty}}\left(\sum_{n\ge1}\frac{q^{2n}}{1-q^{2n}}+\sum_{n\ge1}\frac{q^{2n}}{1+q^{2n}}\right)+\sum_{n\ge1}\frac{(q)_{n-1}(-1)^nq^{n(n+1)/2}}{(-q)_{n}},\end{equation}
\rm
We also need a well-known transformation of Fine [12, pg. 5, eq.(6.3)]. 
\\*
\\*
{\bf Lemma 2.} \it We have,
\begin{equation}\sum_{n\ge0}\frac{(aq)_n}{(bq)_n}t^n=\frac{1-b}{1-t}\sum_{n\ge0}\frac{(atq/b)_n}{(tq)_n}b^n. \end{equation}

\rm
We will need also need the special case of Lemma 2 ($a=-q^{-1/2},$ $b=q^{1/2},$ $t=-q,$) coupled with a result due to Lovejoy [15]:
\begin{equation}q\sum_{n\ge0}\frac{(-q;q^2)_n}{(q;q^2)_{n+1}}(-q^2)^n=-\sum_{n\ge1}\frac{(q^2;q^2)_{n-1}}{(-q^2;q^2)_{n}}q^n=-\sum_{n\ge1}\frac{(-q;-q)_{n-1}(-1)^n(-q)^{n(n+1)/2}}{(q;-q)_{n}}=-L_2(-1,-q).\end{equation}
{\bf Lemma 3. [6, pg.30, $b\rightarrow-t,$ $a=xt$]}
\begin{equation}\sum_{n\ge0}\frac{(xq;q^2)_n}{(t;q^2)_{n+1}}(tq)^n=\sum_{n\ge0}\frac{(xq)_{n}t^nq^{n(n+1)/2}}{(t)_{n+1}}.\end{equation}

We are now ready to prove the main $q$-series identities of the paper.
\begin{theorem}
\begin{equation}\sum_{n\ge0}\left(\frac{(-q^2;q^2)_{\infty}}{(q;q^2)_{\infty}}-\omega_n(q)\right)=\frac{(-q^2;q^2)_{\infty}}{(q;q^2)_{\infty}}\left(-\frac{1}{2}+2\sum_{n\ge1}\frac{q^{2n-1}}{1-q^{2n-1}}\right)+\frac{1}{2}\sum_{n\ge0}\frac{(q)_n(-1)^nq^{n(n+1)/2}}{(-q)_{n}},\end{equation}

\begin{equation}\sum_{n\ge0}\left(\frac{(-q;q^2)_{\infty}}{(q;q^2)_{\infty}}-\Theta_n(q)\right)=2\frac{(-q;q^2)_{\infty}}{(q;q^2)_{\infty}}\left(\sum_{n\ge1}\frac{q^{2n-1}}{1-q^{2n-1}}\right)+\sum_{n\ge1}\frac{(q)_{n-1}(-1)^nq^{n(n+1)/2}}{(-q)_{n}}.\end{equation}
\end{theorem}
\begin{proof} In keeping with [4] we shall set $\epsilon$ to be the differential operator $\lim_{t\rightarrow1^{-}}\frac{\partial }{\partial t}.$ Applying Proposition 2.1 of [4] to the left side of (2.3), and putting $x=-t$ in Lemma 3. gives
\begin{align}\epsilon(1-t)\sum_{n\ge0}\omega_n(q)t^n\\
&=\epsilon\sum_{n\ge0}\frac{(-tq)_nt^nq^{n(n+1)/2}}{(tq)_{n}}\\
&=\epsilon(1-t)\sum_{n\ge0}\frac{(-tq;q^2)_n}{(t;q^2)_{n+1}}(tq)^n\\
&=\epsilon(1-t)\sum_{n\ge0}\frac{(-q^2;q^2)_nt^n}{(q;q^2)_{n+1}}+\sum_{n\ge1}\frac{(-q;q^2)_n}{(q^2;q^2)_{n}}nq^n\\
&=\epsilon(1-t)\sum_{n\ge0}\frac{(-q^2;q^2)_nt^n}{(q;q^2)_{n+1}}+\epsilon\frac{(-tq^2;q^2)_{\infty}}{(tq;q^2)_{\infty}}\\
&=\epsilon(1-t)\sum_{n\ge0}\frac{(-q^2;q^2)_nt^n}{(q;q^2)_{n+1}}+\frac{(-q^2;q^2)_{\infty}}{(q;q^2)_{\infty}}\left(\sum_{n\ge1}\frac{q^{2n}}{1+q^{2n}}+\sum_{n\ge1}\frac{q^{2n+1}}{1-q^{2n+1}}\right)
\end{align}
Line (2.17) follows from the specialization of Lemma 2 with $q\rightarrow q^2,$ $b=q,$ $a=-1,$ and standard properties of $\epsilon.$ The final step to prove (2.12), requires invoking Proposition 2.1 of [4] coupled
with (2.7) of Lemma 1. 
\par The proof of (2.13) is very similar, but we use (2.8) of Lemma 1 coupled with (2.4) and (2.10).
\end{proof}
The methods employed in [4, 8] to obtain special values of $L$-functions may be applied to our main theorem. We consider equation (2.13) after noting that $(-q;q^2)_{\infty}/(q;q^2)_{\infty}$
vanishes to infinite order when $q\rightarrow-1.$ Define the real quadratic field $K:=\mathbb{Q}(\sqrt{2}).$
\begin{theorem} \it As a formal power series in $t$ we have,
\begin{equation}-\sum_{n\ge0}\Theta_n(-e^{-t})=\sum_{n\ge0}\frac{(-t)^nL_K(-n)}{n!},\end{equation}
where
\begin{equation}L_K(s)=\sum_{\substack{a\subset\mathbb{Z}[\sqrt{2}]\\ N(a)<0}}\frac{i^{N(a)^2+N(a)+2|N(a)|}}{|N(a)|^s}.\end{equation}
\end{theorem}
 The method to obtain this result can be found in [4], which amounts to using Mellin transforms, and calculating the residue $s=-n$ of $\Gamma(s)L_K(s),$
so we omit the details.
\par Next we consider a combinatorial interpretation of (2.12) of Theorem 1. Let $D_{\le n}$ be the set of partitions into distinct parts $\le n,$ and let $D_{n}$ be the set of partitions into $n$
distinct parts. Let $DE$ be the set of partitions with distinct evens (i.e. even parts do not repeat) [5].
\\*
\begin{theorem} Define $\gamma(\pi_1,\pi_2):=\pi_2'+n(\pi_1),$ where $\pi_2'$ denotes the largest part of $\pi_2.$ Let $n_o(\pi)$ be the number of odd parts in a partition $\pi.$ Then
\begin{equation}\sum_{n\ge0}\sum_{(\pi_1,\pi_2)\in D_{\le n}\times D_{n}}\left(\gamma(\pi_1,\pi_2)+\frac{1-(-1)^{\gamma(\pi_1,\pi_2)}}{2}\right)q^{|\pi_1|+|\pi_2|}=2\sum_{\pi\in DE}n_o(\pi)q^{|\pi|}.\end{equation}
\end{theorem}
\begin{proof} We consider the combinatorial interpretation related to that given in [10, pg. 400]. Note that $(-aq)_n$ generates partitions into distinct 
parts $\le n$ and $a$ keeps track of the number of parts. The function $b^nq^{n(n+1)/2}/(bq)_n$ generates partitions into $n$ distinct partitions
and $b$ keeps track of the largest part. Hence,
\begin{equation}\sum_{n\ge0}\frac{(-aq)_nb^nq^{n(n+1)/2}}{(bq)_n}=\sum_{n\ge0}\sum_{(\pi_1,\pi_2)\in D_{\le n}\times D_{n}}a^{n(\pi_1)}b^{\pi_2'}q^{|\pi_1|+|\pi_2|}.\end{equation}
Putting $a=b=t$ in (2.23) and differentiating with respect to $t$ and setting $t=1$ gives the left side of (2.12). On the other hand, putting $a=b=-1$ in (2.23) gives the far right 
hand side of (2.12). The proof is complete after noting that differentiating 
\begin{equation}\frac{(-q^2;q^2)_{\infty}}{(tq;q^2)_{\infty}}=\sum_{\pi\in DE}t^{n_o(\pi)}q^{|\pi|},\end{equation}
with respect to $t$ and then setting $t=1$ gives the middle term in (2.12).
\end{proof}
\rm
\section{Some remarks for further research}
It is well-known that [4]
\begin{equation}\sum_{n\ge0}\left(\frac{1}{(q)_{\infty}}-\frac{1}{(q)_n}\right)=\frac{1}{(q)_{\infty}}\sum_{n\ge1}\frac{q^n}{1-q^n}.\end{equation}
The left side is $\epsilon(1-t)/(t)_{\infty},$ which is $\sum_{m\ge1}mp(n,m)q^n,$ where $p(n,m)$ is the number of partitions of $n$ into $m$ parts. Since
the number of partitions of $n$ into $m$ parts is equal to the number of partitions of $n$ with largest part $m,$ we may write
\begin{equation}\sum_{n\ge0}\left(\frac{1}{(q)_{\infty}}-j_n(q)\right)=\frac{2}{(q)_{\infty}}\sum_{n\ge1}\frac{q^n}{1-q^n},\end{equation}
where $j_n(q)$ is the number of partitions wherein the largest plus the number of parts is at most $n.$ In fact, Andrews has studied $j_n(q)$ in [3],
and is given by, 
\begin{equation} j_{n}(q)=\sum_{0\le m\le n-1}q^m{n-1 \brack m}_q,\end{equation}
for $n>0,$ and $j_{0}(q)=1.$
This leads us to consider the series [3]
$$M(q,t)=\sum_{n\ge0}\frac{t^{2n}q^{n^2}}{(t)_{n+1}(tq)_n},$$
which has different properties than those listed in the first section. Furthermore, if we subtract $2\sum_{m\ge1}mp(m,n)$ from both sides of (3.2), we find $\sum_{m}mN(m,n)=0.$ This is a well-known result concerning $N(m,n),$ the number of partitions of $n$ with rank $m.$ 
\par At this point it is clear $M(q,t)$ does not appear to satisfy any of (i)--(iii). Indeed, we would need to replace $(ii)$ with $M(q,-1)$ is a mock modular form of weight $1/2,$ and it happens that
$(iii)$ satisfies $\lim_{t\rightarrow1^{-}}\frac{\partial }{\partial t}(1-t)f(t,q)=f(q)D(q),$ which is less interesting without the ``error series" $M(q,-1).$ This is due to the fact that
(i) $\lim_{t\rightarrow1^{-}}(1-t)M(q,t)=M(q)=1/(q)_{\infty}$ is a weight $-1/2$ modular form. In fact, we would 
require a sum of tails over $j_{2n}(q)$ to amend this issue, and satisfy the $(iii)$ in the first section.
\begin{equation}\sum_{n\ge0}\left(\frac{1}{(q)_{\infty}}-j_{2n}(q)\right)=\frac{1}{(q)_{\infty}}\sum_{n\ge1}\frac{q^n}{1-q^n}-\frac{1}{4}\sum_{n\ge0}\frac{q^{n^2}}{(-q)_n^2}.\end{equation}
The last $q$-series on the far right side of (3.4) is one of Ramanujan's third order mock theta functions [3], which generates $\sum_{m}(-1)^mN(m,n).$

1390 Bumps River Rd. \\*
Centerville, MA
02632 \\*
USA \\*
E-mail: alexpatk@hotmail.com
\end{document}